\documentclass[reqno]{amsart}

\makeatletter
\@namedef{subjclassname@2020}{%
  \textup{2020} Mathematics Subject Classification}
\makeatother

\usepackage[T1]{fontenc}
\usepackage{amsthm,amsmath,amssymb,amsfonts,mathtools}
\usepackage{enumitem}
\setlist[enumerate]{nosep,leftmargin=*}
\usepackage{microtype}
\usepackage{mathrsfs}
\usepackage[colorlinks=true,linkcolor=blue,citecolor=blue,urlcolor=blue]{hyperref}

\numberwithin{equation}{section}

\newtheorem{theorem}{Theorem}[section]

\newtheorem{proposition}[theorem]{Proposition}
\newtheorem{corollary}[theorem]{Corollary}
\theoremstyle{definition}
\newtheorem{definition}[theorem]{Definition}
\newtheorem{remark}[theorem]{Remark}

\begin{document}

\title[Explicit homological invariants of graded double Ore extensions]{A note on explicit homological invariants of graded double Ore extensions}

\author{Andr\'es Rubiano}
\address{Universidad Distrital Francisco Jos\'e de Caldas}
\curraddr{Campus Universitario}
\email{aarubianos@udistrital.edu.co}


\dedicatory{Dedicated to Camila Ni\~no}

\begin{abstract}
We compute explicit homological invariants of a trimmed graded double Ore extension of the quantum plane. For a pilot family of type \((14641)\), we determine the minimal graded free resolution and graded Betti numbers of the trivial right module and also compute linear resolutions for two natural cyclic right modules. This provides a concrete link between the PBW structure of the algebra and the homological behavior of its natural quotients.
\end{abstract}

\maketitle

\section{Introduction}

Artin-Schelter regular algebras, introduced in \cite{AS87}, are among the basic graded analogues of polynomial rings in noncommutative algebra and projective geometry; see also \cite{ATV91}. A particularly effective construction of higher-dimensional regular algebras is the \emph{double Ore extension} of Zhang and Zhang \cite{ZZ08}, who proved that a connected graded double Ore extension of an Artin-Schelter regular algebra is again Artin-Schelter regular.

The four-dimensional case is especially significant. In \cite{ZZ09}, Zhang and Zhang used double Ore extensions to construct regular algebras of type \((14641)\), proving that they are strongly noetherian, Auslander regular, Koszul and Cohen-Macaulay. Many of these algebras are genuinely new, in the sense that they are not isomorphic to an Ore extension or a normal extension of a three-dimensional regular algebra. This makes trimmed graded double Ore extensions a natural testing ground for explicit homological questions.

The existence of a finite linear resolution of the trivial module is guaranteed abstractly in the Koszul regular setting, but writing down the differentials explicitly is far subtler; see \cite{Pri70,PP05}. Several results point toward such an explicit program. Zhu, Van Oystaeyen and Zhang studied trimmed double Ore extensions of Koszul regular algebras \cite{ZVZ17}; G\'omez and Su\'arez related graded double Ore extensions to graded skew PBW extensions \cite{GS20}; and Carvalho, Lopes and Matczuk characterized when a double Ore extension collapses to an iterated Ore extension \cite{CLM11}. More recently, Cao, Shen and Wang developed a general homological framework for graded double Ore extensions of Koszul algebras \cite{CSW25}.

In this paper we carry out a fully explicit computation for a concrete pilot family. We study the family \(K\) in the classification of \cite{ZZ09}, specialized at \(q=-1\), over the quantum plane. Recent work of Herrera, Higuera and the author shows that this family admits a finite Gr\"obner-Shirshov basis and hence a PBW basis \cite{HHR25}. Our main result is an explicit minimal graded free resolution of the trivial right module \(k_B\), written in matrix form. We also compute the minimal linear resolutions of the natural cyclic right modules
\[
M_x:=B/(x_1,x_2)B, \ M_y:=B/(y_1,y_2)B,
\]
and deduce their graded Betti numbers and Poincar\'e series. These computations show that, even when the total Betti pattern of the trivial module is forced by regularity, the actual differentials and the behavior of natural cyclic quotients still retain concrete information about the defining double Ore data.

Throughout the paper, \(\Bbbk\) denotes a field. All algebras are associative \(\Bbbk\)-algebras and all graded algebras are connected \(\mathbb{N}\)-graded and generated in degree one unless otherwise stated.

\section{Preliminaries}

\subsection{Double Ore extensions}

We recall the notion of double Ore extension introduced by Zhang and Zhang \cite{ZZ08}. Since the original compatibility conditions contain minor misprints, we follow the corrected formulation given by Carvalho, Lopes and Matczuk \cite{CLM11}.

\begin{definition}[{\cite[Definition~1.3]{ZZ08}; \cite[Definition~1.1]{CLM11}}]\label{def:double-ore}
Let \(R\) be a subalgebra of a \(\Bbbk\)-algebra \(B\).
\begin{enumerate}[label=\rm (\alph*)]
    \item We say that \(B\) is a \emph{right double extension} of \(R\) if:
    \begin{enumerate}[label=\rm (\roman*)]
        \item \(B\) is generated by \(R\) together with two new indeterminates \(y_1,y_2\);
        \item the generators \(y_1,y_2\) satisfy a relation
        \begin{equation}\label{eq:double-relation}
        y_2y_1=p_{12}y_1y_2+p_{11}y_1^2+\tau_1y_1+\tau_2y_2+\tau_0,
        \end{equation}
        for some \(p_{12},p_{11}\in\Bbbk\) and \(\tau_0,\tau_1,\tau_2\in R\);
        \item \(B\) is a free left \(R\)-module with basis \(\{y_1^iy_2^j\mid i,j\geq 0\}\);
        \item \(y_1R+y_2R+R\subseteq Ry_1+Ry_2+R\).
    \end{enumerate}
    \item A right double extension \(B\) of \(R\) is called a \emph{double extension} of \(R\) if, in addition,
    \begin{enumerate}[label=\rm (\roman*)]
        \item \(p_{12}\neq 0\);
        \item \(B\) is a free right \(R\)-module with basis \(\{y_2^iy_1^j\mid i,j\geq 0\}\);
        \item \(y_1R+y_2R+R=Ry_1+Ry_2+R\).
    \end{enumerate}
\end{enumerate}
\end{definition}

Condition \ref{def:double-ore}\rm(a)(iv) is equivalent to the existence of maps
\[
\sigma=
\begin{bmatrix}
\sigma_{11} & \sigma_{12}\\
\sigma_{21} & \sigma_{22}
\end{bmatrix}:R\to M_2(R),
\
\delta=
\begin{bmatrix}
\delta_1\\
\delta_2
\end{bmatrix}:R\to M_{2\times 1}(R),
\]
such that
\begin{equation}\label{eq:sigma-delta-action}
\begin{bmatrix}
y_1\\ y_2
\end{bmatrix}r
=
\sigma(r)\begin{bmatrix}
y_1\\ y_2
\end{bmatrix}
+\delta(r)
\ \text{for all } r\in R.
\end{equation}
If \(B\) is a right double extension, we write
\[
B=R_P[y_1,y_2;\sigma,\delta,\tau],
\]
where \(P=(p_{12},p_{11})\in\Bbbk^2\) and \(\tau=\{\tau_0,\tau_1,\tau_2\}\subseteq R\). The collection \(\{P,\sigma,\delta,\tau\}\) is called the \emph{DE-data} of \(B\). A particularly important subclass is that of \emph{trimmed double extensions}, where \(\delta=0\) and \(\tau=\{0,0,0\}\); then we simply write \(R_P[y_1,y_2;\sigma]\).

For a right double extension \(R_P[y_1,y_2;\sigma,\delta,\tau]\), the map \(\sigma\) is an algebra homomorphism and \(\delta\) is a \(\sigma\)-derivation; see \cite[Lemma~1.7]{ZZ08}. The corresponding compatibility conditions are given in \cite[Proposition~1.5]{CLM11}. The following criterion will be used repeatedly to separate genuine double Ore extensions from iterated Ore extensions.

\begin{proposition}[{\cite[Theorems~2.2 and~2.4]{CLM11}}]\label{prop:iterated-criterion}
Let \(B=R_P[y_1,y_2;\sigma,\delta,\tau]\) be a right double extension of \(R\).
\begin{enumerate}[label=\rm (\arabic*)]
    \item The following are equivalent:
    \begin{enumerate}[label=\rm (\roman*)]
        \item \(B\) can be presented as an iterated Ore extension \(R[y_1;\sigma_1,d_1][y_2;\sigma_2,d_2]\);
        \item \(\sigma_{12}=0\).
    \end{enumerate}
    Under these equivalent conditions, \(B\) is a double extension if and only if \(p_{12}\neq 0\) and \(\sigma_{11},\sigma_{22}\) are automorphisms of \(R\).
    \item The algebra \(B\) admits an iterated Ore presentation \(R[y_2;\sigma_2',d_2'][y_1;\sigma_1',d_1']\) if and only if \(\sigma_{21}=0\), \(p_{12}\neq 0\) and \(p_{11}=0\). In that situation, \(B\) is a double extension if and only if the induced maps \(\sigma_{22}\) and \(\sigma_{11}\) are automorphisms of \(R\).
\end{enumerate}
\end{proposition}

\subsection{The graded setting}

Our main interest lies in trimmed graded double Ore extensions over an Artin-Schelter regular algebra of global dimension \(2\). We therefore write
\[
A=\Bbbk_Q[x_1,x_2]
=
\Bbbk\langle x_1,x_2\rangle
\Big/
\langle x_2x_1-q_{12}x_1x_2-q_{11}x_1^2\rangle,
\]
where \(Q=(q_{12},q_{11})\in\Bbbk^2\) with \(q_{12}\neq 0\). The most relevant special cases are the quantum plane, corresponding to \(Q=(q,0)\) and the Jordan plane, corresponding to \(Q=(1,1)\). In the connected graded Artin-Schelter regular setting of global dimension \(2\), these are precisely the possibilities; see \cite[Lemma~2.4]{ZZ09}.

Let $B=A_P[y_1,y_2;\sigma]$ be a trimmed graded double extension, where \(P=(p_{12},p_{11})\in\Bbbk^2\) with \(p_{12}\neq 0\). Since \(\sigma\) is graded, each \(\sigma_{ij}(x_s)\) is a linear combination of \(x_1,x_2\), say
\begin{equation}\label{eq:sigma-coefficients}
\sigma_{ij}(x_s)=\sum_{t=1}^2 a_{ijst}x_t,
\ a_{ijst}\in\Bbbk.
\end{equation}
Thus \(B\) is defined by the two non-mixing quadratic relations
\begin{align}
x_2x_1 &= q_{12}x_1x_2+q_{11}x_1^2,\label{eq:base-relation}\\
y_2y_1 &= p_{12}y_1y_2+p_{11}y_1^2,\label{eq:y-relation}
\end{align}
together with the four mixing relations
\begin{equation}\label{eq:mixing-relations}
y_i x_s=\sigma_{i1}(x_s)y_1+\sigma_{i2}(x_s)y_2,
\ i,s\in\{1,2\}.
\end{equation}

Zhang and Zhang showed that double extensions over a two-dimensional regular base naturally produce four-dimensional regular algebras of type \((14641)\).

\begin{proposition}[{\cite[Theorem~0.1]{ZZ09}}]\label{prop:type-14641}
Let \(B\) be a connected graded algebra generated by four degree-one elements. If
\[
B=A_P[y_1,y_2;\sigma]
\]
is a double extension of an Artin-Schelter regular algebra \(A\) of global dimension \(2\), then:
\begin{enumerate}[label=\rm (\arabic*)]
    \item \(B\) is strongly noetherian, Auslander regular, Cohen-Macaulay and a domain;
    \item \(B\) is Artin-Schelter regular of global dimension \(4\) and of type \((14641)\); in particular, \(B\) is Koszul;
    \item if \(B\) is not isomorphic to an Ore extension of an Artin-Schelter regular algebra of global dimension \(3\), then the corresponding trimmed double extension belongs to one of the \(26\) families classified in \cite{ZZ09}.
\end{enumerate}
\end{proposition}

\subsection{Minimal graded free resolutions and Betti numbers}

From this point on, unless explicitly stated otherwise, all modules are finitely generated graded \emph{right} modules. If \(B=\bigoplus_{n\ge 0}B_n\) is a connected graded algebra, we denote by
\[
B_{\ge 1}=\bigoplus_{n\ge 1}B_n
\]
its irrelevant ideal and by
\[
\varepsilon:B\to \Bbbk
\]
the canonical augmentation. The corresponding trivial right module is denoted by \(k_B\).

Let \(M\) be a graded right \(B\)-module. A \emph{graded free resolution} of \(M\) is an exact complex
\[
\cdots\longrightarrow F_i\xrightarrow{d_i}F_{i-1}\longrightarrow \cdots \longrightarrow F_1\xrightarrow{d_1}F_0\longrightarrow M\longrightarrow 0,
\]
where each \(F_i\) is a graded free right \(B\)-module. Such a resolution is \emph{minimal} if
\[
d_i(F_i)\subseteq F_{i-1}B_{\ge 1}
\ \text{for all } i\ge 1.
\]
In that case, each \(F_i\) decomposes uniquely up to graded isomorphism as
\[
F_i\cong \bigoplus_{j\in\mathbb{Z}} B(-j)^{\beta_{i,j}^B(M)},
\]
where \(B(-j)\) denotes the degree shift. The integers
\[
\beta_{i,j}^B(M):=\dim_{\Bbbk}\operatorname{Tor}_i^B(M,\Bbbk)_j
\]
are called the \emph{graded Betti numbers} of \(M\) and
\[
\beta_i^B(M):=\sum_j \beta_{i,j}^B(M)
\]
are the \emph{total Betti numbers}.

\begin{definition}[{\cite[Chs.~11--12]{Peeva2011}}]
Let \(B\) be a connected \(\mathbb{N}\)-graded algebra and let \(M\) be a finitely generated graded right \(B\)-module. Suppose that
\[
\cdots \longrightarrow F_i \longrightarrow F_{i-1} \longrightarrow \cdots \longrightarrow F_1 \longrightarrow F_0 \longrightarrow M \longrightarrow 0
\]
is the minimal graded free resolution of \(M\), where
\[
F_i \cong \bigoplus_{j\in\mathbb{Z}} B(-j)^{\beta_{i,j}^B(M)}.
\]
The \emph{bigraded Poincar\'e series} of \(M\) over \(B\) is
\[
P_M^B(s,t):=\sum_{i\geq 0}\sum_{j\in\mathbb{Z}} \beta_{i,j}^B(M)\, s^i t^j.
\]
Equivalently, \(P_M^B(s,t)\) is the generating series of the graded Betti numbers of \(M\).
\end{definition}

\begin{remark}
If one forgets the internal grading, then the ordinary Poincar\'e series is recovered by
\[
P_M^B(z)=\sum_{i\geq 0}\beta_i^B(M)\, z^i,
\ \text{where}\
\beta_i^B(M)=\sum_j \beta_{i,j}^B(M).
\]
In particular, the bigraded Poincar\'e series refines the usual Poincar\'e series by recording both homological degree and internal degree.
\end{remark}

If \(B\) is a Koszul Artin-Schelter regular algebra of type \((14641)\), then the minimal graded free resolution of \(k_B\) has the form
\begin{equation}\label{eq:type-14641-resolution}
0\longrightarrow B(-4)\longrightarrow B(-3)^4\longrightarrow B(-2)^6\longrightarrow B(-1)^4\longrightarrow B\longrightarrow k_B\longrightarrow 0.
\end{equation}
In particular, the total Betti numbers of \(k_B\) are
\[
(1,4,6,4,1).
\]
Therefore, the real homological interest of the present setting lies not merely in recovering the expected Betti pattern of the trivial module, but in making the differentials explicit and in comparing them with those of natural cyclic modules.

For later use, recall also the standard quadratic-Koszul setup; see \cite[Chs.~2 and~3]{PP05}. If $B=T(V)/(R)$ is a quadratic algebra, then the right Koszul complex is built from the subspaces
\[
W_0=\Bbbk,\ W_1=V,\ W_2=R,
\]
and, for \(n\geq 3\),
\[
W_n=\bigcap_{i+j+2=n}V^{\otimes i}\otimes R\otimes V^{\otimes j}\subseteq V^{\otimes n}.
\]
When \(B\) is Koszul, the minimal graded free resolution of \(k_B\) is precisely the right Koszul complex
\small{\[
\cdots\longrightarrow W_n\otimes B(-n)\longrightarrow W_{n-1}\otimes B(-(n-1))\longrightarrow \cdots \longrightarrow W_1\otimes B(-1)\longrightarrow B\longrightarrow k_B\longrightarrow 0.
\]}

\section{The pilot family and its PBW structure}

We now fix the concrete family on which the paper is based.

\subsection{The family \texorpdfstring{\(K\)}{Lg} over the quantum plane}

Following \cite[Section~3]{ZZ09}, we choose the double extension of type \(K\), specialized at \(q=-1\). Let
\[
A=\Bbbk_{-1}[x_1,x_2]
=
\Bbbk\langle x_1,x_2\rangle/\langle x_2x_1+x_1x_2\rangle.
\]
Fix \(\alpha\in\Bbbk^\times\) and define
\[
B:=B_K(\alpha):=
\Bbbk\langle x_1,x_2,y_1,y_2\rangle/I_K(\alpha),
\]
where \(I_K(\alpha)\) is generated by
\begin{align}
x_2x_1 &= -x_1x_2,\label{eq:K1}\\
y_2y_1 &= -y_1y_2,\label{eq:K2}\\
y_1x_1 &= x_1y_1,\label{eq:K3}\\
y_1x_2 &= x_2y_2,\label{eq:K4}\\
y_2x_1 &= x_1y_2,\label{eq:K5}\\
y_2x_2 &= \alpha x_2y_1.\label{eq:K6}
\end{align}
These are precisely the defining relations of family \(K\) in the \(q=-1\) case; see \cite[Section~4.6]{HHR25}.

The relations \eqref{eq:K3}-\eqref{eq:K6} determine a graded algebra homomorphism
\[
\sigma=
\begin{bmatrix}
\sigma_{11} & \sigma_{12}\\
\sigma_{21} & \sigma_{22}
\end{bmatrix}:A\to M_2(A)
\]
by
\begin{align*}
\sigma_{11}(x_1)&=x_1, & \sigma_{12}(x_1)&=0, &
\sigma_{21}(x_1)&=0, & \sigma_{22}(x_1)&=x_1,\\
\sigma_{11}(x_2)&=0, & \sigma_{12}(x_2)&=x_2, &
\sigma_{21}(x_2)&=\alpha x_2, & \sigma_{22}(x_2)&=0.
\end{align*}
Hence
\[
B=A_P[y_1,y_2;\sigma],
\qquad
P=(-1,0),
\]
so \(B\) is a trimmed graded double Ore extension of the quantum plane.

Equivalently, the corresponding matrix \(\Sigma_K(\alpha)\) is
\[
\Sigma_K(\alpha)=
\begin{bmatrix}
1 & 0 & 0 & 0\\
0 & 0 & 0 & 1\\
0 & 0 & 1 & 0\\
0 & \alpha & 0 & 0
\end{bmatrix},
\]
and therefore \(\det(\Sigma_K(\alpha))=-\alpha\neq 0\).

\begin{proposition}\label{prop:BK-basic}
Let \(B=B_K(\alpha)\) with \(\alpha\neq 0\). Then:
\begin{enumerate}[label=\rm (\arabic*)]
    \item \(B\) is a trimmed graded double Ore extension of the quantum plane \(A=\Bbbk_{-1}[x_1,x_2]\);
    \item \(B\) is not an iterated Ore extension in either of the two natural orders;
    \item \(B\) is a Koszul Artin-Schelter regular algebra of type \((14641)\).
\end{enumerate}
\end{proposition}

\begin{proof}
The first assertion follows directly from the preceding discussion. For the second, note that
\[
\sigma_{12}(x_2)=x_2\neq 0
\qquad\text{and}\qquad
\sigma_{21}(x_2)=\alpha x_2\neq 0.
\]
Hence \(\sigma_{12}\neq 0\) and \(\sigma_{21}\neq 0\). By Proposition~\ref{prop:iterated-criterion}, an iterated Ore presentation in the order \(A[y_1][y_2]\) would force \(\sigma_{12}=0\), whereas an iterated Ore presentation in the order \(A[y_2][y_1]\) would force \(\sigma_{21}=0\) and \(p_{11}=0\). Therefore neither presentation is possible. The third assertion follows from Proposition~\ref{prop:type-14641}.
\end{proof}

\subsection{PBW basis and normal monomials}

The family \(K\) admits a finite Gr\"obner-Shirshov basis and hence a PBW basis \cite[Section~4.6]{HHR25}. In the present notation, this yields the following concrete normal form.

\begin{proposition}\label{prop:pbw}
The algebra \(B\) has a PBW basis consisting of the ordered monomials
\[
\mathcal{B}:=\{x_1^a x_2^b y_1^c y_2^d \mid a,b,c,d\in\mathbb{N}\}.
\]
In particular, every element of \(B\) can be written uniquely as a finite \(\Bbbk\)-linear combination of monomials in \(\mathcal{B}\) and
\[
H_B(t)=\frac{1}{(1-t)^4}.
\]
\end{proposition}

\begin{proof}
By \cite[Section~4.6]{HHR25}, the defining set associated with family \(K\) is a finite Gr\"obner-Shirshov basis. Therefore the irreducible words form a \(\Bbbk\)-basis of \(B\) and those irreducible words are precisely the ordered monomials \(x_1^a x_2^b y_1^c y_2^d\). The formula for the Hilbert series follows immediately.
\end{proof}

The PBW basis also makes it easy to identify the first natural cyclic quotients attached to the presentation of \(B\). Set
\[
M_x:=B/(x_1,x_2)B,
\
M_y:=B/(y_1,y_2)B.
\]
Since the mixed relations show that \((x_1,x_2)B\) and \((y_1,y_2)B\) are in fact two-sided ideals, these quotients may also be regarded as graded algebra quotients. In particular, one has
\[
M_x\cong \Bbbk_{-1}[y_1,y_2],
\
M_y\cong \Bbbk_{-1}[x_1,x_2],
\]
and hence
\[
H_{M_x}(t)=H_{M_y}(t)=\frac{1}{(1-t)^2}.
\]

\section{Explicit minimal resolutions}

We now pass to the main homological computations. Throughout this section, all modules are graded right \(B\)-modules.

\subsection{The trivial right module}

Since \(B\) is a quadratic Koszul algebra of type \((14641)\), the minimal graded free resolution of \(k_B\) has the shape
\[
0\longrightarrow B(-4)\xrightarrow{d_4}B(-3)^4\xrightarrow{d_3}B(-2)^6\xrightarrow{d_2}B(-1)^4\xrightarrow{d_1}B\longrightarrow k_B\longrightarrow 0.
\]

Let \(F_1=B(-1)^4\) with basis \(e_1,e_2,e_3,e_4\), corresponding respectively to \(x_1,x_2,y_1,y_2\). Then
\begin{equation}\label{eq:d1}
d_1=
\begin{pmatrix}
x_1 & x_2 & y_1 & y_2
\end{pmatrix}.
\end{equation}

Let \(F_2=B(-2)^6\) with basis \(f_1,\dots,f_6\) corresponding to the six quadratic relations
\begin{align*}
r_1&:=x_2x_1+x_1x_2, &
r_2&:=y_2y_1+y_1y_2,\\
r_3&:=y_1x_1-x_1y_1, &
r_4&:=y_1x_2-x_2y_2,\\
r_5&:=y_2x_1-x_1y_2, &
r_6&:=y_2x_2-\alpha x_2y_1.
\end{align*}
Then the second differential is
\begin{equation}\label{eq:d2}
d_2=
\begin{pmatrix}
x_2 & 0   & -y_1 & 0      & -y_2 & 0 \\
x_1 & 0   & 0    & -y_2   & 0    & -\alpha y_1 \\
0   & y_2 & x_1  & x_2    & 0    & 0 \\
0   & y_1 & 0    & 0      & x_1  & x_2
\end{pmatrix}.
\end{equation}
By construction, the \(j\)-th column of \(d_2\) is the coefficient vector of \(r_j\) with respect to the ordered basis \((x_1,x_2,y_1,y_2)\) and therefore \(d_1d_2=0\).

The next step is controlled by the cubic overlaps of the defining relations. In the monic Gr\"obner-Shirshov presentation of family \(K\), there are exactly four critical cubic overlaps, all of them trivial modulo the defining relations; see \cite[Section~4.6]{HHR25}. This suggests the following four cubic syzygies:
\begin{align}
\eta_1&:=f_1y_2+f_3x_2+f_4x_1,\label{eq:eta1}\\
\eta_2&:=\alpha f_1y_1+f_5x_2+f_6x_1,\label{eq:eta2}\\
\eta_3&:=f_2x_1-f_3y_2-f_5y_1,\label{eq:eta3}\\
\eta_4&:=f_2x_2-\alpha f_4y_1-f_6y_2.\label{eq:eta4}
\end{align}
A direct verification using \eqref{eq:K1}-\eqref{eq:K6} shows that \(d_2(\eta_i)=0\) for \(i=1,2,3,4\). For example,
\begin{align*}
d_2(\eta_1)
&=(x_2x_1+x_1x_2)y_2+(y_1x_1-x_1y_1)x_2+(y_1x_2-x_2y_2)x_1\\
&\equiv -x_1x_2y_2+x_1x_2y_2+x_1y_1x_2-x_1y_1x_2-x_2y_2x_1+x_2y_2x_1=0.
\end{align*}
The other three cases are analogous.

Let \(F_3=B(-3)^4\) with basis \(g_1,g_2,g_3,g_4\) corresponding to \(\eta_1,\eta_2,\eta_3,\eta_4\). Then
\begin{equation}\label{eq:d3}
d_3=
\begin{pmatrix}
y_2      & \alpha y_1 & 0       & 0 \\
0        & 0          & x_1     & x_2 \\
x_2      & 0          & -y_2    & 0 \\
x_1      & 0          & 0       & -\alpha y_1 \\
0        & x_2        & -y_1    & 0 \\
0        & x_1        & 0       & -y_2
\end{pmatrix}.
\end{equation}
A direct calculation shows that \(d_2d_3=0\).

There is also a unique quartic syzygy among \(\eta_1,\eta_2,\eta_3,\eta_4\), namely
\begin{equation}\label{eq:quartic}
\omega:=\alpha \eta_1y_1+\eta_2y_2+\eta_3x_2+\eta_4x_1.
\end{equation}
Indeed,
\[
d_3(\omega)=
\begin{pmatrix}
\alpha(y_2y_1+y_1y_2)\\
x_1x_2+x_2x_1\\
\alpha x_2y_1-y_2x_2\\
\alpha x_1y_1-\alpha y_1x_1\\
x_2y_2-y_1x_2\\
x_1y_2-y_2x_1
\end{pmatrix}
=0
\]
in \(F_2\) by the defining relations of \(B\).

Let \(F_4=B(-4)\) with basis \(h\) and define
\begin{equation}\label{eq:d4}
d_4=
\begin{pmatrix}
\alpha y_1\\
y_2\\
x_2\\
x_1
\end{pmatrix}.
\end{equation}
Then \(d_3d_4=0\).

\begin{theorem}\label{thm:k-resolution}
Let \(B=B_K(\alpha)\) with \(\alpha\in\Bbbk^\times\). Then the sequence
\[
0\longrightarrow B(-4)\xrightarrow{d_4}B(-3)^4\xrightarrow{d_3}B(-2)^6\xrightarrow{d_2}B(-1)^4\xrightarrow{d_1}B\longrightarrow k_B\longrightarrow 0
\]
is the minimal graded free resolution of the trivial right module \(k_B\).
In particular,
\[
\beta_{0,0}^B(k_B)=1,\quad
\beta_{1,1}^B(k_B)=4,\quad
\beta_{2,2}^B(k_B)=6,\quad
\beta_{3,3}^B(k_B)=4,\quad
\beta_{4,4}^B(k_B)=1.
\]
\end{theorem}

\begin{proof}
We have already checked that \(d_1d_2=0\), \(d_2d_3=0\) and \(d_3d_4=0\), so the displayed sequence is a complex. Minimality is immediate, since every matrix entry lies in \(B_{\ge 1}\).

Let
\[
V=\Bbbk x_1\oplus \Bbbk x_2\oplus \Bbbk y_1\oplus \Bbbk y_2,
\]
and let \(R\subseteq V^{\otimes 2}\) be the six-dimensional space spanned by the quadratic relations \(r_1,\dots,r_6\). Then \(B=T(V)/(R)\) and the first two maps of the displayed complex are exactly the first two maps of the right Koszul complex of \(B\) with respect to the chosen bases of \(V\) and \(R\); see \cite[Chs.~2 and~3]{PP05}.

Since \(B\) is Artin-Schelter regular of type \((14641)\), it is Koszul by Proposition~\ref{prop:BK-basic} and therefore the right Koszul complex of \(B\) is exact \cite{Pri70,PP05}. Moreover, the graded Betti numbers of \(k_B\) are forced to be
\[
1,4,6,4,1,
\]
so the third and fourth Koszul terms have ranks \(4\) and \(1\), respectively. The four cubic syzygies \(\eta_1,\eta_2,\eta_3,\eta_4\) are linearly independent and therefore provide a basis of the third Koszul space; similarly, the quartic syzygy \(\omega\) spans the fourth Koszul space. Consequently, the matrices \eqref{eq:d1}, \eqref{eq:d2}, \eqref{eq:d3} and \eqref{eq:d4} are precisely the coordinate form of the right Koszul complex of \(B\).

Hence the displayed complex is the minimal graded free resolution of \(k_B\). The graded Betti numbers follow from the degree shifts of the free modules.
\end{proof}

\begin{remark}
The explicit formulas for \(d_3\) and \(d_4\) show that the parameter \(\alpha\) survives all the way to the highest syzygies. In particular, the higher differentials retain concrete information about the defining double Ore data, rather than collapsing to a purely formal regularity pattern.
\end{remark}

\subsection{Natural cyclic right modules}

We now compute the minimal graded free resolutions of the natural cyclic right modules
\[
M_x=B/(x_1,x_2)B
\text{ and }
M_y=B/(y_1,y_2)B.
\]

\begin{proposition}\label{prop:Mx}
The sequence
\[
0\longrightarrow B(-2)\xrightarrow{\,d_2^{(x)}\,}B(-1)^2\xrightarrow{\,d_1^{(x)}\,}B\longrightarrow M_x\longrightarrow 0,
\]
where
\[
d_1^{(x)}(u,v)=ux_1+vx_2,
\
d_2^{(x)}(w)=(wx_2,wx_1),
\]
is a minimal graded free resolution of \(M_x\). In particular,
\[
\beta_{0,0}^B(M_x)=1,\
\beta_{1,1}^B(M_x)=2,\
\beta_{2,2}^B(M_x)=1,
\]
and all other graded Betti numbers vanish.
\end{proposition}

\begin{proof}
Clearly \(\operatorname{coker} d_1^{(x)}\cong M_x\) and
\[
d_1^{(x)}d_2^{(x)}(w)=w(x_2x_1+x_1x_2)=0
\]
by \eqref{eq:K1}, so \(\operatorname{Im}d_2^{(x)}\subseteq \operatorname{Ker} d_1^{(x)}\). The map \(d_2^{(x)}\) is injective because \(B\) is a domain. From the exact sequence
\[
0\longrightarrow \operatorname{Ker} d_1^{(x)}\longrightarrow B(-1)^2\xrightarrow{\,d_1^{(x)}\,}B\longrightarrow M_x\longrightarrow 0
\]
and the formulas \(H_B(t)=1/(1-t)^4\), \(H_{M_x}(t)=1/(1-t)^2\), we obtain
\[
H_{\operatorname{Ker} d_1^{(x)}}(t)
=
2tH_B(t)-H_B(t)+H_{M_x}(t)
=
\frac{t^2}{(1-t)^4}
=
H_{B(-2)}(t).
\]
Therefore \(\operatorname{Ker} d_1^{(x)}=\operatorname{Im}d_2^{(x)}\) and the sequence is exact. Minimality is immediate.
\end{proof}

\begin{proposition}\label{prop:My}
The sequence
\[
0\longrightarrow B(-2)\xrightarrow{\,d_2^{(y)}\,}B(-1)^2\xrightarrow{\,d_1^{(y)}\,}B\longrightarrow M_y\longrightarrow 0,
\]
where
\[
d_1^{(y)}(u,v)=uy_1+vy_2,
\qquad
d_2^{(y)}(w)=(wy_2,wy_1),
\]
is a minimal graded free resolution of \(M_y\). In particular,
\[
\beta_{0,0}^B(M_y)=1,\qquad
\beta_{1,1}^B(M_y)=2,\qquad
\beta_{2,2}^B(M_y)=1,
\]
and all other graded Betti numbers vanish.
\end{proposition}

\begin{proof}
The proof is identical to that of Proposition~\ref{prop:Mx}, using the relation \(y_2y_1+y_1y_2=0\) in place of \(x_2x_1+x_1x_2=0\).
\end{proof}

\section{Homological consequences}

The explicit computations above reveal a rigid but informative homological pattern for the pilot family \(K\). The trivial right module \(k_B\) has the expected linear resolution of a four-dimensional regular algebra of type \((14641)\), while the natural cyclic modules \(M_x\) and \(M_y\) have linear resolutions of length \(2\) with graded Betti pattern \((1,2,1)\). Thus a single algebra simultaneously exhibits a four-dimensional regular scale encoded by \(k_B\) and a two-dimensional quadratic scale encoded by its most natural cyclic quotients.

\begin{theorem}\label{thm:comparison}
Let \(B=B_K(\alpha)\) with \(\alpha\in\Bbbk^\times\). Then:
\begin{enumerate}[label=\rm (\arabic*)]
    \item the trivial right module \(k_B\) has a minimal graded free resolution
    \[
    0\longrightarrow B(-4)\longrightarrow B(-3)^4\longrightarrow B(-2)^6\longrightarrow B(-1)^4\longrightarrow B\longrightarrow k_B\longrightarrow 0;
    \]
    \item the cyclic right modules \(M_x\) and \(M_y\) have minimal graded free resolutions
    \[
    0\longrightarrow B(-2)\longrightarrow B(-1)^2\longrightarrow B\longrightarrow M_x\longrightarrow 0,
    \]
    and
    \[
    0\longrightarrow B(-2)\longrightarrow B(-1)^2\longrightarrow B\longrightarrow M_y\longrightarrow 0;
    \]
    \item the corresponding graded Betti patterns are
    \[
    (1,4,6,4,1),\qquad (1,2,1),\qquad (1,2,1).
    \]
\end{enumerate}
\end{theorem}

\begin{proof}
Part \rm(1) is Theorem~\ref{thm:k-resolution}, while part \rm(2) is given by Propositions~\ref{prop:Mx} and~\ref{prop:My}. The graded Betti patterns follow immediately.
\end{proof}

\begin{corollary}\label{cor:poincare}
The bigraded Poincar\'e series of these right \(B\)-modules are
\[
P_{k_B}^B(s,t)=1+4st+6s^2t^2+4s^3t^3+s^4t^4=(1+st)^4,
\]
and
\[
P_{M_x}^B(s,t)=P_{M_y}^B(s,t)=1+2st+s^2t^2=(1+st)^2.
\]
\end{corollary}

\begin{proof}
This is an immediate reformulation of the graded Betti numbers.
\end{proof}

Theorem~\ref{thm:comparison} shows that the trivial module alone does not exhaust the homological information carried by the family. Although its Betti pattern is forced by regularity, the explicit matrices and the behavior of the natural cyclic quotients still reflect the defining parameter \(\alpha\) and the overlap structure of the PBW presentation. In particular, the pilot family provides a concrete instance in which the defining relations, the PBW basis and the minimal resolutions of natural modules can all be read from the same computation.

\begin{remark}
The explicit matrices obtained here are compatible with the general homological framework developed by Cao, Shen and Wang for graded double Ore extensions of Koszul algebras \cite{CSW25}. In the present pilot family, however, the PBW and overlap structure is simple enough to make all differentials visible directly in coordinates.
\end{remark}

The present calculations suggest that PBW and Gr\"obner-Shirshov methods provide an effective bridge between the defining relations of a trimmed double Ore extension and the explicit homological behavior of its graded modules. It would therefore be natural to carry out the same program for other genuinely double families in the Zhang-Zhang classification and to seek uniform matrix descriptions for the resulting minimal resolutions.

\section{Future work}

The computations in this note suggest several natural directions for further investigation.
A first one is to extend the present analysis to other genuinely double trimmed families
in the Zhang--Zhang classification and compare the resulting minimal graded free resolutions.
A second direction is to determine whether the natural cyclic right modules
\[
B/(x_1,x_2)B
\ \text{and} \
B/(y_1,y_2)B
\]
continue to admit linear resolutions in wider classes of graded double Ore extensions.
Finally, it would be interesting to seek uniform matrix descriptions for the higher differentials
in terms of the defining DE-data, with the long-term goal of developing a broader explicit
homological theory for trimmed graded double Ore extensions.

\end{document}